\documentclass[a4paper,12pt]{article}
\usepackage{amsfonts,amsthm,amsmath}

\usepackage[mathscr]{eucal}
\usepackage{amsmath}
\usepackage{amsfonts}
\usepackage{amssymb}
\usepackage{footnpag}
\usepackage[dvips]{graphicx}

\theoremstyle{plain}
\newtheorem{thm}{Theorem}[section]
\newtheorem{prop}{Proposition}[section]
\newtheorem{lem}{Lemma}[section]
\newtheorem{cor}{Corollary}[section]

\theoremstyle{definition}
\newtheorem{df}{Definition}[section]
\newtheorem{rem}{Remark}[section]

\newtheorem{prob}{Problem}[section]

\newcommand{\Q}{\mathbb{Q}}

\newcommand{\R}{\mathbb{R}}

\newcommand{\OOO}{\mathcal{O}}

\newcommand{\aaa}{\mathfrak{a}}

\newcommand{\QQ}{\mathbb{Q}}

\newcommand{\ZZ}{\mathbb{Z}}
\newcommand{\RR}{\mathbb{R}}

\newcommand{\NN}{\mathbb{N}}

\newcommand{\e}{\hfill $\Box$}

\begin{document}

\title{On a property of $2$-dimensional integral Euclidean lattices}

\author{Eiichi Bannai\thanks{Graduate School of Mathematics Kyushu University 
Motooka 744 Nishi-ku, Fukuoka, 819-0395 Japan. 
email: bannai@math.kyushu-u.ac.jp} 
and 
Tsuyoshi Miezaki\thanks{
Division of Mathematics, 
Graduate School of Information Sciences, 
Tohoku University, 
6-3-09 Aramaki-Aza-Aoba, Aoba-ku, Sendai 980-8579, Japan, 
miezaki@math.is.tohoku.ac.jp}
}
\date{}
\maketitle
\begin{abstract}
Let $\Lambda$ be any integral lattice in the $2$-dimensional Euclidean space. 
Generalizing the earlier works of Hiroshi Maehara and others,
we prove that for every integer $n>0$, 
there is a circle in the plane $\RR^{2}$ 
that passes through exactly $n$ points of $\Lambda$. 
\end{abstract}
\vspace{-0.15in}

\quad

\section{Introduction}
\markright{Universally concyclic}

We consider the following condition 
on $2$-dimensional lattices $\Lambda \subset \RR^{2}$. 
\begin{df}
If there is a circle in the plane $\R^{2}$ that 
passes through exactly $n$ points of $\Lambda$ for every integer $n>0$, 
then $\Lambda$ is called universally concyclic. 
\end{df}
A lattice generated by $(a, b), (c, d)\in \RR^{2}$, 
$(ad-bc\neq 0)$ is denoted by $\Lambda[(a, b), (c, d)]$. 
In \cite{M2}, Maehara introduced the term ``universally concyclic". 
Then, he and others showed the following results. 
In \cite{S} and \cite{MM}, Schinzel, Maehara and Matsumoto proved that 
$\ZZ^{2}$, that is, $\Lambda[(1, 0), (0, 1)]$ is universally concyclic. 
Moreover let $a$, $b$, $c$, $d\in\ZZ$ be such that $q:=ad-bc$ is a prime 
and $q\equiv 3 \pmod{4}$. 
Then $\Lambda[(a, b), (c, d)]$ is universally concyclic. 
The equilateral triangular lattice 
$\Lambda[(1, 0), (-1/2, \sqrt{-3}/2])]$
and 
rectangular lattice 
$\Lambda[(1, 0), (0, \sqrt{-3}])]$ 
are universally concyclic. 

Let $\ZZ[x]:=\{a+bx\mid a,b\in\ZZ\}$. 
We remark that for a positive integer $d$, 
a lattice $\Lambda[(1, 0), (a, b\sqrt{d})]$ is also given by 
$\ZZ[a+b\sqrt{-d}]$ in the complex plane. 
We define the set $A(k)$ as follows: 
\[
A(k):=\{z\in \ZZ[\sqrt{-3}]\mid \vert z\vert ^{2} =7^{k}\}. 
\]
In \cite{M2}, Maehara proved the following result:
\begin{lem}[cf. \cite{M2}]
$\sharp A(k)=2(k+1)$. 
\end{lem}
Then, Maehara \cite{M2} proposed the following problems: 
\begin{prob}[cf. \cite{M2}]\label{prob:Maehara1}
For every square-free integer $d>1$ and a prime $p$ such that 
$p=x^2+y^2 d$, we have $\sharp\{z\in \ZZ[\sqrt{-d}]\mid 
\vert z\vert ^{2}=p^{k}\}\geq 2(k+1)$ for every $k$. Does equality always hold?
\end{prob}
\begin{prob}[cf. \cite{M2}]\label{prob:Maehara2}
Is $\Lambda[(a, b),(c, d)]$ universally 
concyclic if $a$, $b$, $c$, $d\in \ZZ$ 
and $ad-bc\neq 0$. 
\end{prob}
Here, we answer Problems \ref{prob:Maehara1} and \ref{prob:Maehara2} affirmatively. 
In fact, we prove a slightly stronger assertion in Theorems \ref{thm:main1} and \ref{thm:main} below. 
Let $d$ be a square-free positive integer and 
$K$ be the imaginary quadratic field $K=\QQ(\sqrt{-d})$. 
We define $\OOO_{K}$ as the integer ring of $K$. 
Let $\ZZ\cdot a + \ZZ\cdot b$ denote 
the linear combination of $a$ and $b$ with integer coefficients. 
Then $\OOO_{K}$ will be written as follows:
\begin{equation}\label{eqn:L1}
\OOO_{K}=\ZZ\cdot 1+\ZZ\cdot w_K,
\end{equation} 
where
\begin{equation}\label{eqn:L}
w_K=
\left\{
\begin{array}{ll}
\sqrt{-d}\quad &{\rm if}\ -d\equiv 2,\ 3 \pmod {4}, \\
\displaystyle \frac{-1+\sqrt{-d}}{2}\quad &{\rm if}\ -d\equiv 1 \pmod {4}. 
\end{array}
\right.
\end{equation}
We denote by $d_{K}$ the discriminant of $K$:
\[
d_{K}=\left\{
\begin{array}{lll}
-4d\ &{\rm if }\ -d\equiv 2,\ 3&\pmod {4}, \\
-d\ &{\rm if }\ -d\equiv 1 &\pmod {4}. 
\end{array}
\right.
\]
We review the concept of order 
in a quadratic field (for more details, see~\cite{Cox}). 
An order $\OOO$ in a quadratic field $K$ is a subset 
$\OOO\subset K$ such that 
\begin{enumerate}
\item 
$\OOO$ is a subring of $K$ containing $1$. 
\item 
$\OOO$ is a finitely generated $\ZZ$-module. 
\item 
$\OOO$ contains a $\QQ$-basis of $K$. 
\end{enumerate}
We can now describe all orders in a quadratic fields: 
\begin{lem}[cf.~{\cite[page.~133]{Cox}}]\label{lem:order}
Let $\OOO$ be an order in a quadratic field $K$ of discriminant $d_{K}$. 
Then $\OOO$ has a finite index in $\OOO_{K}$, and 
if we set $f=[\OOO_{K}:\OOO]$, then 
\begin{equation}\label{eqn:order}
\OOO=\ZZ+f\OOO_{K}=\ZZ\cdot 1+\ZZ\cdot fw_{K}, 
\end{equation}
where $w_{K}$ is as in $(\ref{eqn:L})$. 
Here $f$ is called a conductor of the order $\OOO$. 
\end{lem}
We denote $\OOO$ by $\OOO_f$ if $f=[\OOO_{K}:\OOO]$. 
Now, we introduce the concept of proper ideals of an order. 
For any ideal $\aaa$ of $\OOO_{f}$, notice that 
\[
\OOO_{f}\subset\{\beta\in K\mid \beta \aaa\subset \aaa\}
\]
since $\aaa$ is an ideal of $\OOO_{f}$. 
We say that an ideal $\aaa$ of $\OOO_{f}$ 
is proper whenever equality holds, i.e., when 
\[
\OOO_{f}=\{\beta\in K\mid \beta \aaa\subset \aaa\}. 
\] 
A quadratic form $F$ is called integral 
if all the coefficients of $F$ are rational integers. 
A lattice $\Lambda$ is called integral 
if $(x, y) \in \ZZ$ for all $x, y \in\Lambda$, 
where $(x, y)$ is the standard inner product. 
Generally, it is well-known that there exists a one-to-one correspondence 
between the set of proper ideal classes of the order 
$\OOO_{f}$ and 
the equivalence class of primitive positive definite integral quadratic forms 
$F(x, y)$ 
with discriminant $f^{2}d_{K}<0$ 
(see Theorem \ref{thm:Cox1} in Section \ref{section:Preliminaries}, 
\cite[Chapter 2, \S 7-6]{BS}, \cite[\S 11]{Zagier}). 
Hence, we consider the proper ideal classes of $\OOO_{f}$ 
to be the lattice in $\RR^2$ corresponding to a quadratic forms $F(x, y)$. 
On the other hand, any $2$-dimensional integral Euclidean lattice 
can be considered as some proper ideal class of $\OOO_{f}$. 
We define $\Lambda$ as the proper ideal classes of $\OOO_{f}$. 
Then, we prove the following theorems: 
\begin{thm}\label{thm:main1}
Let $n\in \NN$ and assume that $n \neq 1$. 
Let $p$ be a prime number such that there exists a $z\in \ZZ[\sqrt{-n}]$ 
with $\vert z\vert ^2 =p$, $\left(\frac{d_K}{p}\right)=1$ and $(p,f)=1$, 
where $\left(\frac{\cdot}{\cdot}\right)$ is the Legendre symbol. 
Then, 
\[\sharp \{z\in \ZZ[\sqrt{-n}] \mid \vert z\vert ^2=p^k\} =2(k+1). 
\]
\end{thm}
\begin{thm}\label{thm:main}
All the $2$-dimensional integral lattices in $\RR^{2}$ 
are universally concyclic. 
\end{thm}
\begin{rem}
We remark that there exist some non-integral lattices 
which are not universally concyclic. 
Maehara also proved in \cite{M2} that 
if $\tau$ is a transcendental number, then $\Lambda[(1, \tau), (0, 1)]$ 
cannot contain four concyclic points, hence is not universally concyclic. 
The rectangular lattice $\Lambda[(\alpha, 0), (0, \beta)]$ does not contain 
five concyclic points if and only if $(\alpha /\beta )^{2}$ is an irrational number. 
Hence, some additional integrality 
conditions are necessary to ensure this property. 
\end{rem}

\section{Preliminaries}\label{section:Preliminaries}
In this paper, 
we consider the $2$-dimensional integral Euclidean lattices. 
We shall always assume that $d$ denotes a positive
square-free integer. 
Let $K=\Q(\sqrt{-d})$ be an imaginary quadratic field, 
and let $\mathcal{O}_{K}$ be its ring of algebraic integers 
defined by (\ref{eqn:L1}). 
As we mentioned in Section 1, 
there exists a one-to-one correspondence 
between the set of fractional ideal classes of the unique quadratic field 
$\QQ(\sqrt{-d})$ and 
the equivalence class of 
primitive positive definite integral quadratic forms $F(x, y)$ 
with discriminant $d_{K}<0$ \cite[\S 10]{Zagier}. 
More generally, there exists a one-to-one correspondence 
between the set of fractional proper ideal classes 
of order $\OOO_{f}$ and 
the equivalence class of primitive positive definite integral quadratic forms 
$F(x, y)$ 
with discriminant $f^{2}d_{K}<0$ \cite[Chapter 2, \S 7-6]{BS}, \cite[\S 11]{Zagier}. 
We remark that the value $f^{2}d_{K}$ is called the discriminant of 
the order $\mathcal{O}_f$. 
Finally, we give the well-known theorems needed later. 
\begin{thm}[cf.~{\cite[page~104]{Cox}}]\label{thm:facprime}
We can classify prime ideals of a quadratic field as follows{\rm :} 
\begin{enumerate}
\item If $p$ is an odd prime and $\left(\frac{d_{K}}{p}\right)=1$ 
$($resp. $d_{K}\equiv 1 \pmod{8}$$)$ then 
\[
(p)=\mathfrak{p} \mathfrak{p}^{\prime}\ (resp.\ (2)=\mathfrak{p} \mathfrak{p}^{\prime}), 
\]
where $\mathfrak{p}$ and $\mathfrak{p}^{\prime}$ are prime ideals with $\mathfrak{p}\neq \mathfrak{p}^{\prime}$, 
$N(\mathfrak{p})=N(\mathfrak{p}^{\prime})=p$ $($resp. $N(\mathfrak{p})=2$$)$. 
\item If $p$ is an odd prime and $\left(\frac{d_{K}}{p}\right)=-1$ 
$($resp. $d_{K}\equiv 5 \pmod{8}$$)$ then 
\[
(p)=\mathfrak{p}\ (resp.\ (2)=\mathfrak{p}), 
\]
where $\mathfrak{p}$ is a prime ideal with $N(\mathfrak{p})=p^{2}$ $($resp. $N(\mathfrak{p})=4$$)$. 
\item If $p\ \vert\ d_{k}$ then 
\[
(p)=\mathfrak{p}^{2}, 
\]
where $\mathfrak{p}$ is a prime ideal with $N(\mathfrak{p})=p$. 
\end{enumerate}
\end{thm}


\begin{thm}[cf.~{\cite[Theorem~7.7]{Cox}}]\label{thm:Cox1}
Let $\mathcal{O}$ be an order of discriminant $D$ 
in an imaginary quadratic field $K$. 
\begin{enumerate}
\item 
If $F(x,y)=ax^2+bxy+cy^2$ is a primitive positive definite integral 
quadratic form of discriminant $D$, then 
$[a, (-b+\sqrt{D})/2]$ is a proper ideal of $\mathcal{O}$. 
\item 
The map sending $F(x,y)$ to $[a,(-b+\sqrt{D})/2]$ induces 
an isomorphism between the form class group 
and 
the ideal class group. 
\item 
A positive integer $m$ is represented by a form 
$F(x,y)$ if and only if $m$ is the norm $N(\mathfrak{a})$ 
of some ideal $\mathfrak{a}$ in the corresponding ideal class 
mentioned in 2. 
\end{enumerate}
\end{thm}

\begin{lem}[cf.~{\cite[Lemma~7.18]{Cox}}]\label{lem:Cox}
Let $\mathcal{O}_f$ be an order of conductor $f$. 
We say that a non-zero $\mathcal{O}_f$-ideal 
$\mathfrak{a}$ is prime to $f$ provided that $\mathfrak{a}+f\mathcal{O}_f=\mathcal{O}_f$. 
\begin{enumerate}
\item 
An $\mathcal{O}_f$-ideal $\mathfrak{a}$ is prime to $f$ 
if and only if its norm $N(\mathfrak{a})$ is 
relatively prime to $f$. 
\item 
Every $\mathcal{O}_f$-ideal prime to $f$ is proper. 
\end{enumerate}
\end{lem}

\begin{prop}[cf.~{\cite[Proposition~7.20]{Cox}}]\label{prop:cox1}
Let $\mathcal{O}_f$ be an order of conductor $f$ 
in an imaginary quadratic field $K$. 
We say that a non-zero $\mathcal{O}_K$-ideal 
$\mathfrak{a}$ is prime to $f$ provided that $\mathfrak{a}+f\mathcal{O}_K=\mathcal{O}_K$. 
If $\mathfrak{a}$ is an $\mathcal{O}_{K}$-ideal prime to $f$, 
then $\mathfrak{a}\cap\mathcal{O}_f$ is an $\mathcal{O}_f$-ideal prime to 
$f$ of the same norm. 

\end{prop}
\begin{prop}[cf.~{\cite[Exercise~7.26]{Cox}}]\label{prop:cox2}
Let $\mathcal{O}_f$ be an order of conductor $f$. 
Then $\mathcal{O}_f$-ideals prime to 
the conductor can be factored uniquely into prime $\mathcal{O}_f$-ideals 
$($which are also prime to $f$$)$. 
\end{prop}



\begin{thm}[cf.~{\cite[Theorem~9.4]{Cox}}]\label{thm:Cox2}
Let $n>0$ be an integer, and $L$ be the 
ring class field of the order $\ZZ[\sqrt{-n}]$ 
in the imaginary quadratic field $K=\QQ(\sqrt{-n})$. 
If $p$ is an odd prime not dividing $n$, 
then 
\[
p=x^2+ny^2\Leftrightarrow p {\it \ splits\ completely\ in\ }L. 
\] 
\end{thm}

\section{Proof of Theorem \ref{thm:main1}}

{\it Proof of Theorem \ref{thm:main1}.} 
%
We remark that $\ZZ[\sqrt{-n}]$ can be considered as 
the order $\ZZ[\sqrt{-n}]=\OOO_{f}\subset K=\QQ(\sqrt{-d})$ for 
some $f$ and $d$ with the following condition $-4n=f^2d_K$, namely, 
\[
n=
\left\{
\begin{array}{ll}
f^2d \quad &{\rm if}\ -d\equiv 2,\ 3 \pmod {4}, \\
\displaystyle \frac{f^2d}{4} \quad &{\rm if}\ -d\equiv 1 \pmod {4}. 
\end{array}
\right.
\]
Therefore, we remark that $\ZZ[\sqrt{-n}]=\OOO_{f}$. 

We fix a prime $p$ such that there exists a $z\in\ZZ[\sqrt{-n}]$ 
with $\vert z\vert^{2}=p$, $\left(\frac{d_{K}}{p}\right)=1$ and $(p,f)=1$. 
Because of Theorem \ref{thm:facprime}, 
$(p)=\mathfrak{p}\mathfrak{p}^{\prime}$ in $\mathcal{O}_{K}$ 
for some $\mathfrak{p}$. 
Moreover, the condition $z\in \ZZ[\sqrt{-n}]$ implies that 
the ideals $\mathfrak{p}$ and $\mathfrak{p}^{\prime}$ are principal ideals. 
We set 
\begin{eqnarray*}
\mathfrak{q}&=&\mathfrak{p}\cap\mathcal{O}_f\\
\mathfrak{q}^{\prime}&=&\mathfrak{p}^{\prime}\cap\mathcal{O}_f. 
\end{eqnarray*}
Then, by Proposition \ref{prop:cox1}, the ideals $\mathfrak{q}$ and 
$\mathfrak{q}^{\prime}$ are 
principal ideals of $\mathcal{O}_{f}$ prime to $f$. 
Because of Lemma \ref{lem:Cox}, $\mathcal{O}_f$-ideal prime to $f$ is proper 
and using the unique factorization of proper ideals 
in Proposition \ref{prop:cox2}, the ideals of norm $p^{k}$ are as follows: 
\begin{equation}\label{eqn:O}
\mathfrak{q}^{k},\ \mathfrak{q}^{k-1}\mathfrak{q}^{\prime},\ldots, 
{\mathfrak{q}^{\prime}}^{k}. 
\end{equation}
Let $z_1$ be the element of $\ZZ[\sqrt{-n}]$ with norm $p^k$. 
Because of Lemma \ref{lem:Cox}, $(z_1)$ is a proper $\OOO_f$-ideal. 
Moreover, for $-z_1\in \ZZ[\sqrt{-n}]$, 
the ideals $(z_1)$ and $(-z_1)$ are same proper $\OOO_f$-ideals. 
Hence, there exists a one-to-one correspondence 
between the non-equivalent elements of $\ZZ[\sqrt{-n}]$ with norm $p^{k}$ 
under the action of $\{\pm1\}$ 
and the set of proper $\OOO_f$-ideals of norm $p^{k}$ 
defined by (\ref{eqn:O}). 
This completes the proof of Theorem \ref{thm:main1}. \e

\section{Proof of Theorem \ref{thm:main}}
\subsection{Setup}

\begin{prop}\label{prop:main}
For any positive integers $n$ and $a$, 
there exists a prime $p$ prime to $n$ such that 
\[p=x^2+ny^2\] 
with $y\equiv 0\pmod{4a}$. 
\end{prop}
\begin{proof}
We set $n^{\prime}=16a^2n$. 
Let $L$ be the ring class field of the order $\ZZ[\sqrt{-n}]$. 
(We refer to Cox \cite{Cox} for the concept of ring class fields.)
Because of Theorem \ref{thm:Cox2}, 
there exists a prime $p$ such that 
\begin{eqnarray*}
p&=&x^2+n^{\prime}y^2\\
&=&x^2+n(4ay)^2
\end{eqnarray*}
if and only if $p$ splits completely in $L$. 
Then 
the primes that split completely in $L$ 
have density $1/[L:K]$, and in particular 
there are infinitely many of them 
(cf. [2. Corollary~5.21] and [2. Corollary~8.18]). 
Hence, there exists a prime $p$ prime to $n$. 
Therefore, we complete the proof of Proposition \ref{prop:main}. 
\end{proof}

Because of Proposition \ref{prop:main}, 
there exists prime $p$ prime to n such that 
$p=x_1^2+ny_1^2$ 
with $y_1\equiv 0\pmod{4a}$. 
We fix such a prime and denote it by $p_{n,a}$. 
Then we define $A_{n,a}(k)$ as follows: 
\[
A_{n,a}(k):=\{z\in \ZZ[\sqrt{-n}] \mid \vert z\vert^2=p_{n,a}^k\}. 
\]
By Proposition \ref{prop:main}, if $x+y\sqrt{-n}\in A_{n,a}(k)$ 
then $y\equiv 0 \pmod{4a}$ and 
\begin{equation}
x+y\equiv \pm j\pmod{4a}, \label{eqn:j}
\end{equation}
where 
$j\equiv x_1^k\pmod{4a}$, $1\leq j\leq 4a-1$.  
So, we define $\check{A}_{n,a}(k)$ as follows:
\[
\check{A}_{n,a}(k):=\{x+y\sqrt{-n}\in A_{n,a}(k)\mid x+y\equiv -j \pmod{4a}\}. 
\]

\begin{lem}\label{lem:number}
$\sharp A_{n,a}(k)=2(k+1)$ and $\sharp\check{A}_{n,a}(k)=k+1$. 
\end{lem}

\begin{proof}
Because of Proposition \ref{prop:main}, 
$(d_K/p_{n,a})=1$ and $(p_{n,a},f)=1$. 
Hence, by Theorem \ref{thm:main1} $\sharp A_{n,a}(k)=2(k+1)$. 
If $x+y\sqrt{-n}\in {A}_{n,a}(k)$, 
then $x\neq 0$, $-x+y\sqrt{-n}\in {A}_{n,a}(k)$, 
and only one of them belongs to $\check{A}_{n,a}(k)$. 
Therefore, $\sharp\check{A}_{n,a}(k)=k+1$. 
\end{proof}


\subsection{Proof of Theorem \ref{thm:main}}
Here, we start the proof of Theorem \ref{thm:main}. \\
{\it Proof of Theorem \ref{thm:main}.} 
Let $\Lambda$ be a $2$-dimensional integral lattice 
and let the associated quadratic form be $ax^2+bxy+cy^2$. 
Let $\OOO_{f}\subset \QQ{\sqrt{-d}}$ 
be the order corresponding to the lattice $\Lambda$. 
We set $n=-f^2d_{K}$ and $\alpha:=(-b+\sqrt{-n})/(2\sqrt{a})$. 
It is enough to show that for each integer $k>0$, 
there is a circle in the 
complex plane that passes through exactly $k+1$ points of $\Lambda$. 
For $k>0$, define a circle $\Gamma_{k}$ 
in complex plane as follows: 
\[\vert 4\sqrt{a}z-j\vert ^2=p_{n,a}^k, \] 
where $j$ is defined by (\ref{eqn:j}). 
Let $C(k)$ be the subset of $\Lambda$ lying on the circle  
$\Gamma_{k}$. 
We show that $\sharp C(k)=k+1$. 
If $z=\sqrt{a}x+\alpha y\in C(k)$ then $4\sqrt{a}z-j=4ax-2by-j+2y\sqrt{-n}$, 
so $4ax-2by-j+2y\equiv -j\pmod{4a}$. 
Therefore $4\sqrt{a}z-j\in \check{A}_{n,a}(k)$. 
Hence we can define the map $\varphi:C(k)\rightarrow \check{A}_{n,a}(k)$ by: 
\[
z\mapsto 4\sqrt{a}z-j.
\]
This map is a bijection. 
To see this, suppose $x+y\sqrt{-n}\in \check{A}_{n,a}(k)$. 
Then $x+y\equiv -j\pmod{4a}$, 
that is, $x+by+j \equiv 0\pmod{4a}$. 
Moreover, by Proposition \ref{prop:main}, 
$y\equiv 0 \pmod{4a}$, and hence $y$ is even. 
Therefore, we can define a map from $ \check{A}_{n,a}(k)$ to $C(k)$ as follows: 
\[
x+y\sqrt{-n}\mapsto \frac{x+by+j}{4\sqrt{a}}+\frac{y}{2}\alpha.
\]
This gives the inverse of $\varphi$. Therefore $\varphi$ is surjective, that is, 
$\sharp C(k)=\sharp \check{A}_{n,a}(k)=k+1$. \e

Informing Hiroshi Maehara of Theorem \ref{thm:main}, 
he proved the following fact: 
\begin{cor}\label{cor:Maehara}
If $(\alpha /\beta)^{2}\in \QQ$ 
then $\Lambda[(\alpha, 0), (0, \beta)]$ is universally concyclic.  
\end{cor}
\begin{proof}
We assume that $(\alpha /\beta)^2=b/a$, 
where $b/a$ is irreducible fraction. 
Then, the lattices $\Lambda[(\alpha,0),(0,\beta)]$ and 
$\Lambda[(a,0),(0,\sqrt{ab})]$ are similar 
under the similarity transformation $\alpha/a$ 
and $\Lambda[(a,0),(0,\sqrt{ab})]$ is integral lattice. 
Because of Theorem \ref{thm:main}, 
$\Lambda[(a,0),(0,\sqrt{ab})]$ is universally concyclic, 
so is $\Lambda[(\alpha,0),(0,\beta)]$. 
\end{proof}
\begin{rem}
Finally, we generalize the definition of universally concyclic to 
higher dimensions. 
\begin{df}\label{df:higher}
Let $\Lambda\subset \RR^{d}$ be a $d$-dimensional lattice. 
If there is a spherical surface $S^{d-1}$ in $\RR^d$ 
that passes through exactly $n$ points of $\Lambda$ 
for every integer $n > 0$, 
then $\Lambda$ is called universally concyclic. 
\end{df}
In \cite{M2}, Maehara remark that $\ZZ^{3}$ 
is universally concyclic because the spherical surface 
$(4x-1)^2+(4y)^2+(4z-\sqrt{2})^2 = 17k+2$ 
passes through exactly $k+1$ 
points of $\ZZ^3$. 
We also remark that any integral lattices in higher dimension $\RR^d$ 
are universally concyclic: 
\begin{cor}\label{cor:higher}
Any integral lattices in $\RR^d$ are universally concyclic. 
\end{cor}
\begin{proof}
Let $\Lambda$ be an integral lattice in $\RR^d$. 
We define sublattices $\{\Lambda^{(i)}\}_{i=2}^{d}$ such that 
\[
\Lambda^{(2)}\subset \Lambda^{(3)} \subset \cdots \subset \Lambda^{(d)} = \Lambda
\]
and $\Lambda^{(i)}$ spans $\RR^i$ which we denote by $\RR^{(i)}$ for all $i$. 
Because of Theorem \ref{thm:main}, for each $k>0$, 
we can define the circle $S^{(1)}\subset\RR^{(2)}$ 
that passes through exactly $k$ points of $\Lambda^{(2)}$. 


Let $O^{(1)}$ be the center of $S^{(1)}$ and 
let $\ell$ be a half line in $\RR^{(3)}$ whose origin is $O^{(1)}$, 
which is orthogonal to $\RR^{(2)}$. 
We define the sphere $S^{(2)}(a)$, 
whose center $O^{(2)}(a)$ lies on $\ell$, the 
distance between $O^{(1)}$ and $O^{(2)}(a)$ is $a$ 
and whose radius is $\sqrt{a^2 + (\mbox{radius of } S^{(1)})^2}$. 
We assume that $0\leq a \leq 1$. 

Since $\Lambda$ is an integral lattice, 
the number of the points of $\Lambda^{(3)}$ 
which intersect in $S^{(1)}(a)$ 
is finite for any $0\leq a \leq 1$. 
Moreover, for $a_1\neq a_2$, 
the intersection of $S^{(1)}(a_1)$ and $S^{(1)}(a_2)$ is 
the points of $\Lambda^{(2)}$ in $\Lambda$, namely, the points of $S^{(1)}$. 
On the other hand, for $0\leq a\leq 1$, 
the number of the spheres $S^{(2)}(a)$ is infinite. 
Therefore, there exists a number $a_0$ such that 
the intersection of $S^{(2)}(a_0)$ and $\Lambda$ is 
the points of $\Lambda^{(2)}$. 
We denote $S^{(2)}(a)$ by $S^{(2)}$ 
and $S^{(2)}$ passes through exactly $k$ points of $\Lambda^{(3)}$. 
We can define the spheres $S^{(3)},\ldots , S^{(d-1)}$ recursively 
such that each of $\{S^{(i)}\}_{i=3}^{d-1}$ 
passes through exactly $k$ points of $\Lambda$, 
as we defined $S^{(2)}$ in $\RR^{(3)}$. 
\end{proof}
So, we have shown that 
any integral lattices in $\RR^d$ are universally concyclic. 
However, the points of lattice 
lying on the sphere constructed in the proof of 
Corollary \ref{cor:higher}
are on the plane 
$x_3 = \cdots = x_d = 0$. 
Hence, Maehara added some conditions to Definition \ref{df:higher} 
and showed the following theorem:
\begin{thm}[cf.~\cite{M2}] 
For $n > d \geq 2$, 
there is a sphere in $\RR^d$ that passes through 
exactly $n$ lattice points on $\ZZ^d$, 
and moreover, the $n$ lattice points span 
a $d$-dimensional polytope. 
\end{thm} 
Therefore, we can state the following problem: 
\begin{prob}
Let $\Lambda$ be an integral lattice in $\RR^d$. 
We assume $n > d \geq 2$. 
Is there a sphere in $\RR^d$ that passes through 
exactly $n$ lattice points on $\Lambda$, 
which span a $d$-dimensional polytope? 
\end{prob} 
A set of points in the $d$-dimensional Euclidean space is said to be 
in general position if no $d+1$ of them lie in a $(d-1)$-dimensional plane. 
Then, Maehara also proposed the following problem: 
\begin{prob}[cf.~\cite{M2}]\label{prob:Z^3} 
Is there a sphere in $\RR^3$ that passes through a given 
number of lattice points in general position on $\ZZ^3$? 
\end{prob}
It is also an interesting open problem to prove or disprove a similar 
conclusion as in Problem \ref{prob:Z^3} for any integral lattices 
in higher dimension $\RR^{d}$. 
\end{rem}
\bigskip
\noindent
{\bf Acknowledgement.}
The authors are indebted to Hiroshi Maehara for interesting us to 
the property universally concyclic of Euclidean lattices 
and giving the proof of Corollary \ref{cor:Maehara}. 
The authors also thank 
Masanobu Kaneko of Kyushu University 
for informing us the concept of ring class field, \cite{Cox}, 
and for his help in completing 
the proof of Proposition \ref{prop:main} and 
Kensaku Kinjo of Tohoku University for the careful reading of this manuscript and helpful comments. 
The second author was supported by JSPS Research Fellowship.


\end{document}